\documentclass[12pt]{article}
\usepackage{amsthm,amsmath,amssymb}
\usepackage{graphicx}
\usepackage{booktabs}
\usepackage{geometry}
\usepackage{tikz}

\newtheorem{thm}{Theorem}[section]
\newtheorem{lemma}[thm]{Lemma}
\newtheorem{pro}[thm]{Proposition}
\newtheorem{que}[thm]{Question}
\newtheorem*{claim}{Claim}

\begin{document}

\title{The minimum spectral radius of\\ maximal outerplanar graphs}
\author{Suil O\thanks{Department of Applied Mathematics and Statistics, The State University of New York, Korea, Incheon, 21985, suil.o@sunykorea.ac.kr. Corresponding author. Research supported by the National Research Foundation of Korea (NRF) grant funded by the Korea government(MSIT) No. RS-2025-23523950.}}
\date{}
\maketitle

\begin{abstract}
An outerplanar graph is \emph{maximal} if no edge can be added without losing outerplanarity. Lin and Ning determined the outerplanar graph with the largest spectral radius, and the maximizer is a maximal outerplanar graph. We determine the minimizer. In this paper, we prove that every $n$-vertex maximal outerplanar graph $G$ satisfies $\rho(G)\ge\rho(F_n)$, where $F_n$ is the zig-zag triangulation of the $n$-gon, that is, the square of the path on $n$ vertices, with equality if and only if $G=F_n$. The proof uses three local operations on maximal outerplanar graphs, each of which strictly decreases the spectral radius: the first reverses the way a piece is attached along a chord, and the second and third move a piece from one vertex to its twin across a chord when the twin carries nothing or a single ear, respectively. A graph at which no operation applies is $F_n$, or has spectral radius greater than $4$, or consists of a central triangle with three zig-zag blades of at least three triangles each and has at most $15$ vertices; in the last case it contains one of two explicit graphs on $12$ vertices whose spectral radius exceeds that of $F_{15}$. Since $\rho(F_n)<4$ for all $n$, this completes the proof. The numerical inequalities used along the way are certified by explicit integer vectors with small entries.


\end{abstract}

\section{Introduction}

The \emph{spectral radius} of a graph $G$, written $\rho(G)$, is the largest eigenvalue of its adjacency matrix $A(G)$. In 1986, Brualdi and Solheid~\cite{BrualdiSolheid} raised the problem of maximizing or minimizing the spectral radius over a given family of graphs, and this problem has been studied for many families since then.

A graph is \emph{outerplanar} if it can be drawn in the plane so that all vertices lie on the boundary of the outer face, and an outerplanar graph is \emph{maximal} if no edge can be added without losing outerplanarity. We call such a graph a \emph{maximal outerplanar graph}, or MOP for short. Every MOP on $n\ge3$ vertices has $2n-3$ edges and $n-2$ triangular inner faces, and the boundary of its outer face is a Hamiltonian cycle; see~\cite{Allgeier,HopkinsStaton}.

For outerplanar graphs, Tait and Tobin~\cite{TaitTobin} proved the Cvetkovi\'c--Rowlinson conjecture~\cite{CvetkovicRowlinson} for all sufficiently large $n$, and Lin and Ning~\cite{LinNing} completed it for all $n$: the spectral radius of an $n$-vertex outerplanar graph is maximized uniquely by $K_1\vee P_{n-1}$. Since adding edges increases the spectral radius, this maximizer is a MOP. The minimization problem over all outerplanar graphs is trivial, but over MOPs it is not, since every MOP on $n$ vertices has the same number of edges. The classical model is the theorem of Collatz and Sinogowitz~\cite{CollatzSinogowitz} that the path uniquely minimizes the spectral radius among connected graphs on $n$ vertices. In this paper, we determine the minimizer among MOPs.

For $n\ge3$, let $F_n$ be the MOP obtained from the cycle $v_1v_2\cdots v_nv_1$ by adding the chords
$$v_2v_n,\ v_nv_3,\ v_3v_{n-1},\ v_{n-1}v_4,\ v_4v_{n-2},\ \ldots$$
until the polygon is triangulated (Figure~\ref{fig:Fn}). Listing the vertices in the order $v_1$, $v_2$, $v_n$, $v_3$, $v_{n-1}$, $v_4$, $\ldots$ as $x_1,\ldots,x_n$,
 the inner faces of $F_n$ are the triangles $\{x_i,x_{i+1},x_{i+2}\}$, so $F_n$ is the square $P_n^2$ of the path $x_1x_2\cdots x_n$. For $n\ge5$, its degree sequence is $(4,\ldots,4,3,3,2,2)$.

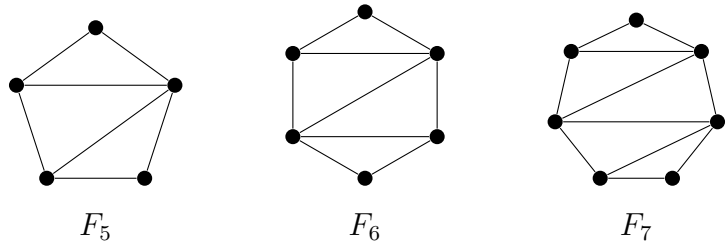
\begin{figure}[h]
\centering
\begin{tabular}{ccc}
\begin{tikzpicture}[every node/.style={circle,fill=black,inner sep=2pt}]
  \foreach \i in {1,...,5} \node (a\i) at ({90+72*(\i-1)}:1.1) {};
  \foreach \i [remember=\i as \j (initially 5)] in {1,...,5} \draw (a\j)--(a\i);
  \draw (a2)--(a5); \draw (a5)--(a3);
\end{tikzpicture}
&\hspace{0.8cm}
\begin{tikzpicture}[every node/.style={circle,fill=black,inner sep=2pt}]
  \foreach \i in {1,...,6} \node (b\i) at ({90+60*(\i-1)}:1.1) {};
  \foreach \i [remember=\i as \j (initially 6)] in {1,...,6} \draw (b\j)--(b\i);
  \draw (b2)--(b6); \draw (b6)--(b3); \draw (b3)--(b5);
\end{tikzpicture}
&\hspace{0.8cm}
\begin{tikzpicture}[every node/.style={circle,fill=black,inner sep=2pt}]
  \foreach \i in {1,...,7} \node (c\i) at ({90+360/7*(\i-1)}:1.1) {};
  \foreach \i [remember=\i as \j (initially 7)] in {1,...,7} \draw (c\j)--(c\i);
  \draw (c2)--(c7); \draw (c7)--(c3); \draw (c3)--(c6); \draw (c6)--(c4);
\end{tikzpicture}
\\[4pt]
$F_5$ &\hspace{0.8cm} $F_6$ &\hspace{0.8cm} $F_7$
\end{tabular}
\caption{The zig-zag triangulations $F_5$, $F_6$, and $F_7$.}
\label{fig:Fn}
\end{figure}

\begin{thm}\label{thm:main}
For every $n$-vertex maximal outerplanar graph $G$,
$$\rho(G)\ge\rho(F_n),$$
with equality if and only if $G=F_n$.
\end{thm}

We now describe the proof. Let $L(G)$ be the number of vertices of degree $2$ in a MOP $G$; it equals the number of leaves of the weak dual tree of $G$. The tools are collected in Section~\ref{sec:tools}. The first is the bound $\rho(F_n)<4$ (Lemma~\ref{lem:band}), which reflects the fact that $F_n$ is a finite section of the square of the two-way infinite path, whose spectrum is the range of $2\cos\theta+2\cos2\theta$ and has maximum $4$. Consequently every MOP with spectral radius at least $4$ already satisfies Theorem~\ref{thm:main}, and only MOPs with spectral radius below $4$ need to be compared with $F_n$. The next three tools are local operations, each of which produces a MOP with strictly smaller spectral radius and keeps the degree-$2$ vertices under control. A chord $st$ cuts a MOP into two pieces, and the \emph{swap} reattaches one piece with the ends $s$ and $t$ exchanged. If $t$ has degree $3$ and $s$ has at least two further neighbors on each side, the swap lowers the spectral radius (Lemma~\ref{lem:swap}); this is proved by a Schur complement onto $\{s,t\}$, which reduces the change of the characteristic polynomial to a product of two positive factors, one from each piece. The swap keeps $L$ fixed. The \emph{ear transfer} applies to a vertex $y$ of degree $2$ whose two neighbors $a,b$ span a chord lying in a triangle $\{a,b,c\}$ all of whose sides are chords. The vertices $y$ and $c$ are then twins with respect to $ab$, and $c$ carries two pieces while $y$ carries none; moving one piece from $c$ to $y$ lowers the spectral radius and decreases $L$ by one (Lemma~\ref{lem:ear}). Here the Perron vector of the new graph, together with the symmetry exchanging $y$ and $c$, decides which of the two pieces is to be moved. The \emph{transfer} treats the case in which the twin $y$ carries a single ear instead of nothing, that is, the triangles beyond $ab$ are $\{a,b,y\}$ and one leaf. Moving the piece of $c$ on the side away from the ear to $y$ lowers the spectral radius, keeps $L$ fixed, and places the ear next to a branch triangle, where an ear transfer applies (Lemma~\ref{lem:transfer}). Heuristically, $c$ keeps its other piece, which is larger than an ear, and a piece attached to the heavier of two twins contributes more; a Schur complement onto $\{a,c,y\}$ makes this precise. Finally, in a MOP where no swap applies, every vertex of degree at least $5$ lies in a triangle all of whose sides are chords, and every path of the other triangles in the weak dual is a zig-zag strip (Lemma~\ref{lem:skeleton}).

In Section~\ref{sec:main}, we start from a MOP $G\ne F_n$ and apply the three operations as long as possible. The spectral radius strictly decreases at every step, and the final graph $G^*$ admits none of them. In $G^*$, the triangles beyond a side of a branch triangle form a zig-zag strip of at least three triangles whenever they contain no further branch triangle. We then show that $G^*=F_n$ if $L(G^*)=2$; that $G^*$ consists of one central triangle with three such strips if $L(G^*)=3$, in which case either $\rho(G^*)>4$ or $G^*$ has at most $15$ vertices and contains a fixed graph whose spectral radius exceeds $\rho(F_{15})$; and that $\rho(G^*)>4$ if $L(G^*)\ge4$. The finitely many inequalities used in the last two cases are certified by integer vectors with entries at most $50$, listed in the Appendix, so that they can be checked by hand; a short script rechecks them in exact arithmetic.

\section{Definitions and Tools}\label{sec:tools}

For terms not defined here we refer to~\cite{West,BrouwerHaemers,GodsilRoyle}. For a graph $H$ we write $f_H(\lambda)=\det(\lambda I-A(H))$; for a subgraph $K$ of a connected graph $G$ we have $\rho(K)\le\rho(G)$, with strict inequality if $K\ne G$, and $f_K(\lambda)>0$ for $\lambda>\rho(K)$.

\begin{lemma}\label{lem:MOP}
Let $G$ be an $n$-vertex MOP with $n\ge3$.
\begin{itemize}
\item[\rm (i)] The boundary of the outer face is a Hamiltonian cycle $C$. The edges of $C$ are the \emph{boundary edges}, and the other $n-3$ edges are \emph{chords}. Every inner face is a triangle, and there are $n-2$ of them.
\item[\rm (ii)] For every vertex $u$ of degree $d$, the neighbors of $u$ induce a path $w_1w_2\cdots w_d$, the triangles containing $u$ are $\{u,w_i,w_{i+1}\}$ for $1\le i\le d-1$, and $uw_1,uw_d$ are boundary edges while $uw_2,\ldots,uw_{d-1}$ are chords.
\item[\rm (iii)] Every chord $st$ lies in exactly two triangles, and it cuts $G$ into two MOPs $G_1$ and $G_2$ with $V(G_1)\cap V(G_2)=\{s,t\}$ and $G=G_1\cup G_2$, in each of which $st$ is a boundary edge. Conversely, gluing two MOPs along a boundary edge of each, with either identification of the ends, gives a MOP.
\end{itemize}
\end{lemma}

The \emph{weak dual} $T(G)$ of a MOP $G$ has the inner triangles as nodes, two being adjacent when they share a chord; it is a tree with $n-2$ nodes and maximum degree at most $3$. We call a triangle a \emph{branch triangle} if all three of its sides are chords, that is, if it has degree $3$ in $T(G)$.

\begin{lemma}\label{lem:dual}
Let $G$ be a MOP with $n\ge4$ vertices.
\begin{itemize}
\item[\rm (i)] For every vertex $v$, the triangles containing $v$ form a path in $T(G)$ with $\deg(v)-1$ nodes.
\item[\rm (ii)] A triangle is a leaf of $T(G)$ if and only if it contains a vertex of degree~$2$, and each leaf contains exactly one such vertex. Hence $T(G)$ has exactly $L(G)$ leaves, and $T(G)$ is a path if and only if $L(G)=2$.
\end{itemize}
\end{lemma}

\begin{proof}
(i) By Lemma~\ref{lem:MOP}(ii), the triangles containing $v$ are $\{v,w_i,w_{i+1}\}$ for $1\le i\le\deg(v)-1$, and consecutive ones share the chord $vw_{i+1}$.

(ii) A triangle is a leaf of $T(G)$ exactly when it has one chord and two boundary edges. If $\{a,b,c\}$ has chord $bc$ and boundary edges $ab,ac$, then no chord is incident with $a$ inside this triangle, so by (i) the vertex $a$ lies in this triangle only and $\deg(a)=2$; the vertices $b,c$ lie on the chord $bc$ and so have degree at least $3$. Conversely, a vertex $a$ of degree $2$ lies in a single triangle by (i), whose two sides at $a$ are boundary edges, so that triangle is a leaf. A tree whose maximum degree is at most $3$ is a path exactly when it has two leaves.
\end{proof}

\begin{pro}\label{pro:Fn}
Let $n\ge5$. If an $n$-vertex MOP has degree sequence $(4,\ldots,4,3,3,2,2)$, then it is $F_n$.
\end{pro}

\begin{proof}
Let $G$ be such a MOP. By Lemma~\ref{lem:dual}, $T(G)$ is a path $T_1T_2\cdots T_m$ with $m=n-2\ge3$, and every vertex $v$ lies in a block of $\deg(v)-1\le3$ consecutive triangles; we say that $v$ \emph{enters} at the first triangle of its block and \emph{retires} at the last. For $1\le i\le m-1$ let $c_i$ be the chord shared by $T_i$ and $T_{i+1}$. For $2\le i\le m-1$, the chords $c_{i-1}$ and $c_i$ are distinct sides of $T_i$ and meet in one vertex; writing $T_i=\{p,x,y\}$ with $c_{i-1}=\{p,x\}$ and $c_i=\{p,y\}$, the vertex $x$ retires at $T_i$ and $y$ enters at $T_i$.

The leaf $T_1$ consists of a vertex of degree $2$ and the chord $c_1=\{a,b\}$. Let $a_3$ enter at $T_2$, so that $T_2=\{a,b,a_3\}$. One of $a,b$ retires at $T_2$, say $b$; then $\deg(b)=3$. We show by induction on $i\ge3$ that the three vertices of $T_i$ entered at $T_{i-2}$, $T_{i-1}$, and $T_i$. This holds for $T_3=\{a,a_3,a_4\}$. If it holds for $T_i$, then the vertex of $T_i$ that entered at $T_{i-2}$ already lies in three triangles, so it retires at $T_i$, and $T_{i+1}$ consists of the vertices that entered at $T_{i-1}$, $T_i$, and $T_{i+1}$. Thus, up to the choice of $b$ at $T_2$, which is a reflection, every step removes the oldest vertex and joins the new vertex to the two remaining ones. This is the construction of $F_n$.
\end{proof}

\begin{lemma}\label{lem:band}
For every $n\ge3$, $\rho(F_n)<4$.
\end{lemma}

\begin{proof}
Write $F_n=P_n^2$ on $x_1,\ldots,x_n$ as in the Introduction. For a nonzero real vector $z$, the inequality $2z_iz_j\le z_i^2+z_j^2$ gives
$$z^{\mathsf T}A(F_n)z=2\sum_{i=1}^{n-1}z_iz_{i+1}+2\sum_{i=1}^{n-2}z_iz_{i+2}\le4\sum_{i=1}^nz_i^2-\bigl(2z_1^2+z_2^2+z_{n-1}^2+2z_n^2\bigr)\le4\sum_{i=1}^nz_i^2 .$$
Equality throughout would force $z_1=z_2=z_{n-1}=z_n=0$ and $z_i=z_{i+1}$ for all $i$, that is, $z=0$. Hence $z^{\mathsf T}A(F_n)z<4z^{\mathsf T}z$ for every $z\ne0$.
\end{proof}

\begin{figure}[t]
\centering
\tikzset{v/.style={circle,fill=black,inner sep=1.4pt},
         piece/.style={fill=gray!25,draw=black,thin},
         ear/.style={fill=gray!10,draw=black,thin},
         lab/.style={font=\footnotesize}}
\newcommand{\core}{%
  \draw (A)--(B); \draw (A)--(P); \draw (B)--(P); \draw (A)--(Q2); \draw (B)--(Q2);}
\begin{tabular}{c}
\begin{tikzpicture}[scale=0.62]
 \coordinate (s) at (0,1); \coordinate (t) at (0,-1); \coordinate (a1) at (-1.2,0); \coordinate (a2) at (1.2,0);
 \fill[piece] (s)--(a1)--(-2.1,1.1)--(-1.1,2.1)--cycle; \node[lab] at (-1.2,1.25) {$G_1$};
 \fill[piece] (s)--(a2)--(2.1,1.1)--(1.1,2.1)--cycle; \node[lab] at (1.2,1.25) {$G_2$};
 \draw (s)--(t) (s)--(a1)--(t)--(a2)--(s);
 \foreach \p in {s,t,a1,a2} \node[v] at (\p) {};
 \node[lab,above] at (s) {$s$}; \node[lab,below] at (t) {$t$}; \node[lab,left] at (a1) {$a_1$}; \node[lab,right] at (a2) {$a_2$};
 \draw[->,thick] (2.6,0)--(3.5,0);
 \begin{scope}[xshift=6.1cm]
 \coordinate (s) at (0,1); \coordinate (t) at (0,-1); \coordinate (a1) at (-1.2,0); \coordinate (a2) at (1.2,0);
 \fill[piece] (s)--(a1)--(-2.1,1.1)--(-1.1,2.1)--cycle; \node[lab] at (-1.2,1.25) {$G_1$};
 \fill[piece] (t)--(a2)--(2.1,-1.1)--(1.1,-2.1)--cycle; \node[lab] at (1.2,-1.25) {$\pi(G_2)$};
 \draw (s)--(t) (s)--(a1)--(t)--(a2)--(s);
 \foreach \p in {s,t,a1,a2} \node[v] at (\p) {};
 \node[lab,above] at (s) {$s$}; \node[lab,below] at (t) {$t$}; \node[lab,left] at (a1) {$a_1$}; \node[lab,right] at (a2) {$a_2$};
 \end{scope}
\end{tikzpicture}\\[-2pt]
{\small (a) Swap at $(s,t)$, where $\deg(t)=3$.}\\[10pt]
\begin{tikzpicture}[scale=0.62]
 \coordinate (a) at (0,1); \coordinate (b) at (0,-1); \coordinate (y) at (-1.2,0); \coordinate (c) at (1.2,0);
 \fill[piece] (c)--(a)--(1.1,2.1)--(2.1,1.1)--cycle; \node[lab] at (1.2,1.25) {$G_Q$};
 \fill[piece] (c)--(b)--(1.1,-2.1)--(2.1,-1.1)--cycle; \node[lab] at (1.2,-1.25) {$G_R$};
 \draw (a)--(b) (a)--(y)--(b)--(c)--(a);
 \foreach \p in {a,b,y,c} \node[v] at (\p) {};
 \node[lab,above] at (a) {$a$}; \node[lab,below] at (b) {$b$}; \node[lab,left] at (y) {$y$}; \node[lab,right] at (c) {$c$};
 \draw[->,thick] (2.6,0)--(3.5,0);
 \begin{scope}[xshift=6.1cm]
 \coordinate (a) at (0,1); \coordinate (b) at (0,-1); \coordinate (y) at (-1.2,0); \coordinate (c) at (1.2,0);
 \fill[piece] (c)--(a)--(1.1,2.1)--(2.1,1.1)--cycle; \node[lab] at (1.2,1.25) {$G_Q$};
 \fill[piece] (y)--(b)--(-1.1,-2.1)--(-2.1,-1.1)--cycle; \node[lab] at (-1.2,-1.25) {$G_R$};
 \draw (a)--(b) (a)--(y)--(b)--(c)--(a);
 \foreach \p in {a,b,y,c} \node[v] at (\p) {};
 \node[lab,above] at (a) {$a$}; \node[lab,below] at (b) {$b$}; \node[lab,left] at (y) {$y$}; \node[lab,right] at (c) {$c$};
 \end{scope}
\end{tikzpicture}\\[-2pt]
{\small (b) Ear transfer at $y$, where $\deg(y)=2$ and $\{a,b,c\}$ is a branch triangle.}\\[10pt]
\begin{tikzpicture}[scale=0.62]
 \coordinate (a) at (0,1); \coordinate (b) at (0,-1); \coordinate (y) at (-1.2,0); \coordinate (c) at (1.2,0); \coordinate (z) at (-1.3,1.5);
 \fill[piece] (c)--(a)--(1.1,2.1)--(2.1,1.1)--cycle; \node[lab] at (1.2,1.25) {$G_Q$};
 \fill[piece] (c)--(b)--(1.1,-2.1)--(2.1,-1.1)--cycle; \node[lab] at (1.2,-1.25) {$G_R$};
 \fill[ear] (a)--(y)--(z)--cycle;
 \draw (a)--(b) (a)--(y)--(b)--(c)--(a) (a)--(z)--(y);
 \foreach \p in {a,b,y,c,z} \node[v] at (\p) {};
 \node[lab,above] at (a) {$a$}; \node[lab,below] at (b) {$b$}; \node[lab,left] at (y) {$y$}; \node[lab,right] at (c) {$c$}; \node[lab,left] at (z) {$z$};
 \draw[->,thick] (2.6,0)--(3.5,0);
 \begin{scope}[xshift=6.1cm]
 \coordinate (a) at (0,1); \coordinate (b) at (0,-1); \coordinate (y) at (-1.2,0); \coordinate (c) at (1.2,0); \coordinate (z) at (-1.3,1.5);
 \fill[piece] (c)--(a)--(1.1,2.1)--(2.1,1.1)--cycle; \node[lab] at (1.2,1.25) {$G_Q$};
 \fill[piece] (y)--(b)--(-1.1,-2.1)--(-2.1,-1.1)--cycle; \node[lab] at (-1.2,-1.25) {$G_R$};
 \fill[ear] (a)--(y)--(z)--cycle;
 \draw (a)--(b) (a)--(y)--(b)--(c)--(a) (a)--(z)--(y);
 \foreach \p in {a,b,y,c,z} \node[v] at (\p) {};
 \node[lab,above] at (a) {$a$}; \node[lab,below] at (b) {$b$}; \node[lab,left] at (y) {$y$}; \node[lab,right] at (c) {$c$}; \node[lab,left] at (z) {$z$};
 \end{scope}
\end{tikzpicture}\\[-2pt]
{\small (c) Transfer at $y$, where $\deg(y)=3$, $\deg(z)=2$, and $\{a,b,c\}$ is a branch triangle.}
\end{tabular}
\caption{The three operations, each shown before and after. Shaded regions are the parts of the graph cut off by the chords shown; in~(a), $G_j$ consists of the triangle $\{s,t,a_j\}$ together with the shaded region beyond it. The piece that moves is $G_2$ in~(a) and $G_R$ in~(b) and~(c). In~(b), $L$ decreases by one. In~(c), $\{a,b,y\}$ becomes a branch triangle adjacent to the leaf $\{a,y,z\}$, so that an ear transfer at $z$ becomes available.}
\label{fig:ops}
\end{figure}

The first operation reattaches a piece along a chord (Figure~\ref{fig:ops}(a)). Let $st$ be a chord of a MOP $G$, and let $G_1$ and $G_2$ be the two MOPs into which it cuts $G$, as in Lemma~\ref{lem:MOP}(iii). Let $\pi$ be the bijection of $V(G_2)$ that exchanges $s$ and $t$ and fixes all other vertices. The \emph{swap} of $G$ at $(s,t)$ is the graph $G_1\cup\pi(G_2)$.

\begin{lemma}\label{lem:swap}
Let $st$ be a chord of a MOP $G$ with $\deg_G(t)=3$, and suppose that $s$ has at least two neighbors in $V(G_j)\setminus\{s,t\}$ for $j=1,2$. Then the swap $G^*$ of $G$ at $(s,t)$ is a MOP, the vertices of degree $2$ in $G^*$ are those of $G$, and $\rho(G^*)<\rho(G)$.
\end{lemma}

\begin{proof}
Since $\deg_G(t)=3$ and $t$ lies in a triangle on each side of $st$, in $G_j$ the vertex $t$ has exactly the two neighbors $s$ and $a_j$, where $\{s,t,a_j\}$ is the triangle of $G_j$ on $st$. By Lemma~\ref{lem:MOP}(iii), $G^*$ is a MOP. In passing from $G$ to $G^*$, every vertex of $V(G_2)\setminus\{a_2,s,t\}$ that is adjacent to $s$ becomes adjacent to $t$ instead, while $a_2$ remains adjacent to both $s$ and $t$. Thus only the degrees of $s$ and $t$ change, and both remain at least $3$.

Put $S=\{s,t\}$ and $X_j=G_j-S$. The graphs $X_1$ and $X_2$ are disjoint, and no edge of $G$ or $G^*$ joins them. Let $C_j$ be the $V(X_j)\times S$ adjacency matrix in $G$, and for $\lambda$ larger than $\rho(X_1)$ and $\rho(X_2)$ let $r_j=(\lambda I-A(X_j))^{-1}$ and $T_j=C_j^{\mathsf T}r_jC_j$, a symmetric $2\times2$ matrix indexed by $(s,t)$. With $J$ the adjacency matrix of the edge $st$, the Schur complement onto $S$ gives
$$f_G=f_{X_1}f_{X_2}\det(N-T_2),\qquad f_{G^*}=f_{X_1}f_{X_2}\det(N-\Pi T_2\Pi),\qquad N=\lambda I_2-J-T_1,$$
where $\Pi$ is the permutation matrix exchanging $s$ and $t$. For $2\times2$ matrices, $\det(N-Y)=\det N-\operatorname{tr}(\operatorname{adj}(N)Y)+\det Y$, and $\det(\Pi T_2\Pi)=\det T_2$. Moreover $T_2-\Pi T_2\Pi=\delta_2\operatorname{diag}(1,-1)$ with $\delta_2=(T_2)_{ss}-(T_2)_{tt}$, and $\operatorname{adj}(N)=\bigl(\begin{smallmatrix}N_{tt}&-N_{st}\\-N_{st}&N_{ss}\end{smallmatrix}\bigr)$ with $N_{tt}-N_{ss}=(T_1)_{ss}-(T_1)_{tt}=:\delta_1$. Hence
\begin{equation}\label{eq:swap}
f_{G^*}-f_G=f_{X_1}f_{X_2}\,\delta_1\delta_2,\qquad \delta_j=\sum_{w,w'\in N(s)\cap V(X_j)}(r_j)_{ww'}-\sum_{w,w'\in N(t)\cap V(X_j)}(r_j)_{ww'} .
\end{equation}

Let $\lambda_0=\rho(G)$. Then $f_G(\lambda_0)=0$ and $f_{X_j}(\lambda_0)>0$. The matrix $r_j$ is entrywise nonnegative with positive diagonal, $N(t)\cap V(X_j)=\{a_j\}\subseteq N(s)\cap V(X_j)$, and $N(s)\cap V(X_j)$ contains a vertex $w\ne a_j$; hence $\delta_j(\lambda_0)\ge(r_j)_{ww}>0$, and $f_{G^*}(\lambda_0)>0$ by~\eqref{eq:swap}. The graph $G^*-t$ is a subgraph of $G-t$, so $\rho(G^*-t)<\lambda_0$, and by interlacing the second largest eigenvalue of $G^*$ is less than $\lambda_0$. Hence all roots of $f_{G^*}$ other than $\rho(G^*)$ lie below $\lambda_0$, and $f_{G^*}(\lambda_0)>0$ forces $\rho(G^*)<\lambda_0$.
\end{proof}

The second operation lowers $L$ (Figure~\ref{fig:ops}(b)). Let $y$ be a vertex of degree $2$ with neighbors $a$ and $b$, and suppose that the other triangle $\{a,b,c\}$ on the chord $ab$ is a branch triangle. The chords $ca$ and $cb$ cut off MOPs $G_Q\ni c,a$ and $G_R\ni c,b$ that do not contain $y$; put $Q^\circ=G_Q-\{a,c\}$ and $R^\circ=G_R-\{b,c\}$, which are nonempty. The \emph{ear transfer} at $y$ replaces every edge $cw$ with $w\in V(R^\circ)$ by the edge $yw$.

\begin{lemma}\label{lem:ear}
Let $G'$ be obtained from a MOP $G$ by an ear transfer at $y$ as above. Then $G'$ is a MOP, $L(G')=L(G)-1$, and $\rho(G')<\rho(G)$.
\end{lemma}

\begin{proof}
The graph $G'$ is obtained by gluing $G_R$ to the boundary edge $yb$ of the MOP $G-V(R^\circ)$, with $c$ placed on $y$; by Lemma~\ref{lem:MOP}(iii) it is a MOP. Only the degrees of $y$ and $c$ change. In $G'$ the vertex $y$ is adjacent to $a$, $b$, and at least one vertex of $R^\circ$, and $c$ is adjacent to $a$, $b$, and at least one vertex of $Q^\circ$. Thus $y$ is no longer of degree $2$ and no new vertex of degree $2$ appears, so $L(G')=L(G)-1$.

Let $u$ be the positive unit Perron vector of $G'$. Write $W_R=N_G(c)\cap V(R^\circ)$ and $W_Q=N_G(c)\cap V(Q^\circ)$; both are nonempty. In $G'$, the vertex $c$ is joined to $W_Q$ and $y$ is joined to $W_R$.

Suppose first that $u_c\ge u_y$. Since $G$ is obtained from $G'$ by moving the edges from $W_R$ back from $y$ to $c$,
$$\rho(G)\ge u^{\mathsf T}A(G)u=u^{\mathsf T}A(G')u+2(u_c-u_y)\sum_{w\in W_R}u_w\ge\rho(G').$$
If equality held, $u$ would maximize the Rayleigh quotient of $A(G)$ and so satisfy $A(G)u=\rho(G)u$; together with $A(G')u=\rho(G')u$ this gives $(A(G)-A(G'))u=0$, whose coordinate at $c$ equals $\sum_{w\in W_R}u_w>0$. Hence $\rho(G)>\rho(G')$.

Suppose now that $u_y>u_c$. Since $N_G(y)=\{a,b\}\subseteq N_G(c)$ and $yc\notin E(G)$, exchanging the names of $y$ and $c$ turns $G$ into an isomorphic graph $\widehat G$, in which $y$ is joined to $a$, $b$, $W_Q$, and $W_R$, and $c$ is joined to $a$ and $b$ only. Then $G'$ is obtained from $\widehat G$ by moving the edges from $W_Q$ from $y$ to $c$, and the same computation gives $\rho(G)=\rho(\widehat G)\ge\rho(G')+2(u_y-u_c)\sum_{w\in W_Q}u_w>\rho(G')$.
\end{proof}

The third operation removes an arm of two triangles (Figure~\ref{fig:ops}(c)). Let $\{a,b,c\}$ be a branch triangle, and let $\{a,b,y\}$ be the other triangle on $ab$. Suppose that $\deg(y)=3$ and that the third neighbor $z$ of $y$ has degree $2$ and is adjacent to $a$; thus the triangles beyond $ab$ are $\{a,b,y\}$ and the leaf $\{a,y,z\}$. As before, the chords $ca$ and $cb$ cut off MOPs $G_Q\ni c,a$ and $G_R\ni c,b$ that do not contain $y$, and we put $Q^\circ=G_Q-\{a,c\}$ and $R^\circ=G_R-\{b,c\}$. The \emph{transfer} at $y$ replaces every edge $cw$ with $w\in V(R^\circ)$ by the edge $yw$; that is, it moves the piece of $c$ on the side away from the ear to $y$.

\begin{lemma}\label{lem:transfer}
Let $G'$ be obtained from a MOP $G$ by a transfer at $y$ as above, and suppose that $|V(Q^\circ)|\ge2$. Then $G'$ is a MOP, $L(G')=L(G)$, the triangle $\{a,b,y\}$ is a branch triangle of $G'$ adjacent to the leaf $\{a,y,z\}$, and $\rho(G')<\rho(G)$.
\end{lemma}

\begin{proof}
In $G$ the vertex $y$ lies only in the triangles $\{a,b,y\}$ and $\{a,y,z\}$, so $yb$ is a boundary edge of $G$ and of $G-V(R^\circ)$. The graph $G'$ is obtained by gluing $G_R$ to $yb$ with $c$ placed on $y$, so it is a MOP by Lemma~\ref{lem:MOP}(iii). Only the degrees of $c$ and $y$ change: $c$ keeps $a$, $b$, and at least one vertex of $Q^\circ$, and $y$ gains at least one vertex of $R^\circ$. Hence $L(G')=L(G)$. In $G'$ all three sides of $\{a,b,y\}$ are chords, and $\{a,y,z\}$ is the triangle across $ay$.

The vertices $c$ and $y$ are twins with respect to $ab$; $c$ carries the piece $G_Q$, while $y$ carries only an ear, and $G_Q$ is larger than an ear. The inequality says that $G_R$ contributes more to the spectral radius when it is attached to the heavier twin. We make this precise.

Let $\lambda=\rho(G)$, and write $W_Q=N_G(c)\cap V(Q^\circ)$, $W_R=N_G(c)\cap V(R^\circ)$, and $N_a=N_G(a)\cap V(Q^\circ)$. Exchanging the names of $c$ and $y$ in $G'$ gives an isomorphic graph $G''$ in which $N(y)=\{a,b\}\cup W_Q$, $N(c)=\{a,b,z\}\cup W_R$, and $N(z)=\{a,c\}$. Thus $G$ and $G''$ differ only in that $G_Q$ is attached along $ac$ and the ear along $ay$ in $G$, and the other way around in $G''$. Put $S=\{a,c,y\}$, $X_1=G[V(Q^\circ)]$, $X_2=G[\{z\}]$, and $X_3=G[\{b\}\cup V(R^\circ)]$. These graphs are the same in $G$ and $G''$, no edge joins two of them, and $S$ spans exactly the edges $ac$ and $ay$ in both graphs. Each $X_i$ is a proper subgraph of the connected graph $G$, so $\lambda>\rho(X_i)$, $f_{X_i}(\lambda)>0$, and $r_i=(\lambda I-A(X_i))^{-1}$ is entrywise nonnegative. Set
$$\alpha=\sum_{w,w'\in N_a}(r_1)_{ww'},\qquad \beta=\sum_{w\in W_Q,\ w'\in N_a}(r_1)_{ww'},\qquad \gamma=\sum_{w,w'\in W_Q}(r_1)_{ww'},\qquad e=\frac1\lambda,$$
$$p=(r_3)_{bb},\qquad q=\sum_{w\in W_R}(r_3)_{bw},\qquad s=\sum_{w,w'\in W_R}(r_3)_{ww'}.$$
The vertex $b$ is adjacent to $a$, $c$, and $y$, and among the vertices of $S$ only $c$ has neighbors in $R^\circ$. Hence the Schur complement onto $S$ gives $f_G(\lambda)=f_{X_1}f_{X_2}f_{X_3}\det M$ and $f_{G''}(\lambda)=f_{X_1}f_{X_2}f_{X_3}\det M''$, where, with rows and columns indexed by $a,c,y$,
$$M=\begin{pmatrix}\lambda-\alpha-e-p&-1-\beta-p-q&-1-e-p\\-1-\beta-p-q&\lambda-\gamma-p-2q-s&-p-q\\-1-e-p&-p-q&\lambda-e-p\end{pmatrix},$$
$$M''=\begin{pmatrix}\lambda-\alpha-e-p&-1-e-p-q&-1-\beta-p\\-1-e-p-q&\lambda-e-p-2q-s&-p-q\\-1-\beta-p&-p-q&\lambda-\gamma-p\end{pmatrix}.$$
A direct expansion gives, with $u=\gamma-e$ and $v=\beta-e$,
\begin{align}
\det M-\det M''={}&q\bigl[2u(\alpha-\lambda-1)-2v(\lambda+e+2)-2v^2\bigr]+q^2(2v-u)\nonumber\\
&+s\bigl[u(\alpha+e+p-\lambda)-2v(e+p+1)-v^2\bigr].\label{eq:transfer}
\end{align}

We now determine the signs. The graph $X_1$ is connected, since it contains all vertices of the boundary cycle of $G_Q$ other than $a$ and $c$, and it has at least two vertices. Let $\{a,c,x\}$ be the triangle of $G_Q$ on $ac$; then $x\in W_Q\cap N_a$, and $x$ has a neighbor in $X_1$. Since $r_1=\sum_{k\ge0}A(X_1)^k/\lambda^{k+1}$, we have $(r_1)_{xx}\ge1/\lambda+(A(X_1)^2)_{xx}/\lambda^3>e$, and as $r_1$ is entrywise nonnegative, $\gamma\ge(r_1)_{xx}$ and $\beta\ge(r_1)_{xx}$. Hence $u>0$ and $v>0$. The graph $X_3$ is connected, so $r_3$ is entrywise positive and $p,q,s>0$. The graph $G-\{c,y\}$ consists of $X_1$, $X_2$, $X_3$, and the vertex $a$, whose neighbors in it are the vertices of $N_a$, $z$, and $b$; hence $f_{G-\{c,y\}}(\lambda)=f_{X_1}f_{X_2}f_{X_3}(\lambda-\alpha-e-p)$, and since $G-\{c,y\}$ is a proper subgraph of $G$, we obtain $\lambda>\alpha+e+p$. Similarly, $f_{G_R}(\lambda)=f_{X_3}(\lambda-p-2q-s)>0$ gives $\lambda>p+2q+s$, and in particular $2q<\lambda$.

By these inequalities, the coefficient of $s$ in~\eqref{eq:transfer} is negative, and so is the bracket multiplying $q$. If $2v\le u$, the term $q^2(2v-u)$ is at most $0$. If $2v>u$, then $q^2(2v-u)<\tfrac{\lambda}{2}q(2v-u)$, and the terms containing $q$ add up to less than $q\bigl[u(2\alpha-\tfrac52\lambda-2)-v(\lambda+2e+4)-2v^2\bigr]$, which is negative because $\alpha<\lambda$ and $u,v>0$. In either case $\det M-\det M''<0$. Since $f_G(\lambda)=0$ gives $\det M=0$, we obtain $\det M''>0$, and so $f_{G'}(\lambda)=f_{G''}(\lambda)>0$. Finally, $G'-y$ is obtained from $G-y$ by deleting the edges from $c$ to $W_R$, so $\rho(G'-y)<\lambda$, and by interlacing the second largest eigenvalue of $G'$ is less than $\lambda$. As in the proof of Lemma~\ref{lem:swap}, $f_{G'}(\lambda)>0$ forces $\rho(G')<\lambda$.
\end{proof}

The last tool describes the MOPs at which no swap applies. We call a MOP \emph{swap-free} if no chord $st$ satisfies the hypothesis of Lemma~\ref{lem:swap}.

\begin{lemma}\label{lem:skeleton}
Let $G$ be a swap-free MOP.
\begin{itemize}
\item[\rm (i)] Every vertex of degree $d\ge5$ lies in at least $\lceil (d-4)/2\rceil$ branch triangles. In particular, every vertex lying in no branch triangle has degree at most $4$ and lies in at most three triangles.
\item[\rm (ii)] Let $\Delta$ be a branch triangle, and let $\Delta T_1\cdots T_m$ be a path in $T(G)$ such that $T_1\cap\Delta=\{a,b\}$ and $T_1,\ldots,T_{m-1}$ are not branch triangles. Then each of $a$ and $b$ lies in at most two of $T_1,\ldots,T_m$.
\item[\rm (iii)] Let $T_1,\ldots,T_m$ form a path in $T(G)$ with chords $c_i=T_i\cap T_{i+1}$, and let $\{a,b\}$ be a side of $T_1$ different from $c_1$. If $a$ and $b$ each lie in at most two of the $T_i$, and every other vertex lies in at most three of them, then the vertices of $T_1\cup\cdots\cup T_m$ can be labeled $x_1,\ldots,x_{m+2}$ with $\{x_1,x_2\}=\{a,b\}$ so that $T_i=\{x_i,x_{i+1},x_{i+2}\}$ for all $i$. In particular, $T_1\cup\cdots\cup T_m$ induces $P_{m+2}^2$, in which $x_1x_2$ is the side $ab$.
\end{itemize}
\end{lemma}

\begin{proof}
(i) Let $p$ have degree $d\ge5$, with neighbors $q_1\cdots q_d$ and triangles $\tau_i=\{p,q_i,q_{i+1}\}$ as in Lemma~\ref{lem:MOP}(ii). For $3\le i\le d-2$, the chord $pq_i$ has the neighbors $q_1,\ldots,q_{i-1}$ of $p$ on one side and $q_{i+1},\ldots,q_d$ on the other, at least two on each side. As $G$ is swap-free, $\deg(q_i)\ne3$; since $q_i$ lies in $\tau_{i-1}$ and $\tau_i$, we get $\deg(q_i)\ge4$, so $q_i$ lies in a third triangle, and $q_{i-1}q_i$ or $q_iq_{i+1}$ is a chord. For $2\le j\le d-2$ the triangle $\tau_j$ already has the chords $pq_j$ and $pq_{j+1}$, so $\tau_{i-1}$ or $\tau_i$ is a branch triangle. Thus the branch triangles meet every pair $\{\tau_{i-1},\tau_i\}$ with $3\le i\le d-2$, and there are at least $\lceil (d-4)/2\rceil$ of them. A vertex of degree at most $4$ lies in at most three triangles by Lemma~\ref{lem:dual}(i).

(ii) By Lemma~\ref{lem:dual}(i), the triangles among $T_1,\ldots,T_m$ that contain $a$ are $T_1,\ldots,T_s$ for some $s$. Suppose $s\ge3$. Write $T_1=\{a,b,w\}$; then $c_1=aw$, and since $a\in T_3$, we have $T_2=\{a,w,z\}$ with $c_2=az$ and $T_3=\{a,z,z'\}$. As $2\le s-1\le m-1$, the triangles $T_1$ and $T_2$ are not branch triangles, so their sides $bw$ and $wz$ are boundary edges, so $w$ lies only in $T_1$ and $T_2$ and $\deg(w)=3$. The chord $aw$ has the neighbors $b$ and $c$ of $a$ on the side of $\Delta=\{a,b,c\}$, and the neighbors $z$ and $z'$ on the other side. This is a swap, a contradiction. The same holds for $b$.

(iii) If $m=1$, label $a$ and $b$ as $x_1$ and $x_2$ in either order and the third vertex of $T_1$ as $x_3$. Let $m\ge2$. The chord $c_1$ contains the vertex of $T_1$ outside $\{a,b\}$ and one of $a,b$; call the latter $x_2$, the other $x_1$, and the former $x_3$. Then $T_1=\{x_1,x_2,x_3\}$ and $c_1=\{x_2,x_3\}$. Suppose that $T_i=\{x_i,x_{i+1},x_{i+2}\}$ and $c_i=\{x_{i+1},x_{i+2}\}$ for some $i\le m-2$, and let $x_{i+3}$ be the vertex of $T_{i+1}$ outside $c_i$. The chord $c_{i+1}$ is a side of $T_{i+1}$ other than $c_i$. If it contained $x_{i+1}$, then $x_{i+1}$ would lie in $T_{i+1}$ and $T_{i+2}$; for $i\ge2$ it also lies in $T_{i-1}$ and $T_i$, which gives four triangles, and for $i=1$ the vertex $x_2\in\{a,b\}$ would lie in three. Hence $T_{i+1}=\{x_{i+1},x_{i+2},x_{i+3}\}$ and $c_{i+1}=\{x_{i+2},x_{i+3}\}$, which completes the induction. Finally, let $x_{m+2}$ be the vertex of $T_m$ outside $c_{m-1}=\{x_m,x_{m+1}\}$. By Lemma~\ref{lem:MOP}(iii), each $x_{i+3}$ lies on the side of $c_i$ opposite to $T_1,\ldots,T_i$, so the $x_j$ are distinct.

It remains to see that no other edge of $G$ joins two of the $x_j$. The union $U$ of $T_1,\ldots,T_m$ is a MOP whose boundary edges are boundary edges or chords of $G$, and by Lemma~\ref{lem:MOP}(iii) each such chord cuts off a piece of $G$ that meets $U$ only in the two ends of that chord. Every edge of $G$ outside $U$ lies in one of these pieces, so it does not join two vertices of $U$. Hence $T_1\cup\cdots\cup T_m$ induces $P_{m+2}^2$.
\end{proof}

\section{Main Results}\label{sec:main}

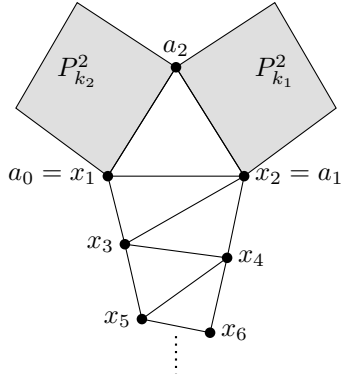
\begin{figure}[t]
\centering
\begin{tikzpicture}[scale=0.9,
   v/.style={circle,fill=black,inner sep=1.4pt}, lab/.style={font=\footnotesize}]
 \coordinate (a0) at (0,0); \coordinate (a1) at (2,0); \coordinate (a2) at (1,1.6);
 \fill[gray!25,draw=black,thin] (a1)--(a2)--(2.45,2.55)--(3.35,1.0)--cycle; \node[lab] at (2.45,1.55) {$P_{k_1}^2$};
 \fill[gray!25,draw=black,thin] (a2)--(a0)--(-1.35,1.0)--(-0.45,2.55)--cycle; \node[lab] at (-0.45,1.55) {$P_{k_2}^2$};
 \coordinate (x3) at (0.25,-1); \coordinate (x4) at (1.75,-1.2); \coordinate (x5) at (0.5,-2.1); \coordinate (x6) at (1.5,-2.3);
 \draw (a0)--(a1)--(a2)--cycle;
 \draw (a0)--(x3)--(x5) (a1)--(x4)--(x6) (a1)--(x3)--(x4)--(x5)--(x6);
 \draw[dotted,thick] (1,-2.35)--(1,-2.9);
 \foreach \p in {a0,a1,a2,x3,x4,x5,x6} \node[v] at (\p) {};
 \node[lab,left] at (a0) {$a_0=x_1$}; \node[lab,right] at (a1) {$x_2=a_1$}; \node[lab,above] at (a2) {$a_2$};
 \node[lab,left] at (x3) {$x_3$}; \node[lab,right] at (x4) {$x_4$}; \node[lab,left] at (x5) {$x_5$}; \node[lab,right] at (x6) {$x_6$};
\end{tikzpicture}
\caption{The graph $D(k;o)$. The blade on $a_0a_1$ is drawn in full; here $o_0=1$, since $x_1$ is identified with $a_0$. The other two blades are glued along $a_1a_2$ and $a_2a_0$ in the same way, with orientations $o_1$ and $o_2$.}
\label{fig:D}
\end{figure}

We first describe the configurations that remain at the end of the reduction. Let $k=(k_0,k_1,k_2)$ with $k_i\ge3$ and $o=(o_0,o_1,o_2)\in\{0,1\}^3$. The graph $D(k;o)$ is obtained from a triangle $a_0a_1a_2$ as follows: for each $i$ (indices modulo $3$), take a copy of $P_{k_i}^2$ on $x_1,\ldots,x_{k_i}$ and identify $(x_1,x_2)$ with $(a_i,a_{i+1})$ if $o_i=1$ and with $(a_{i+1},a_i)$ if $o_i=0$. It is a MOP with $k_0+k_1+k_2-3$ vertices, three vertices of degree $2$, and the single branch triangle $a_0a_1a_2$; we call the three copies of $P_{k_i}^2$ its \emph{blades} (Figure~\ref{fig:D}). If $k'\le k$ coordinatewise, then $D(k';o)$ is a subgraph of $D(k;o)$, since $P_{k_i'}^2$ is induced by the first $k_i'$ vertices of $P_{k_i}^2$; hence $\rho(D(k;o))$ is nondecreasing in each $k_i$.

For configurations with two branch triangles, let $b\ge0$ and $o\in\{0,1\}^4$. The graph $L(b;o)$ consists of two triangles $a_0a_1a_2$ and $a_0'a_1'a_2'$ and a copy of $P_{b+2}^2$ on $y_1,\ldots,y_{b+2}$, called the \emph{bridge}, where $(y_1,y_2)$ is identified with $(a_0,a_1)$ or $(a_1,a_0)$ according as $o_0=1$ or $0$, and $(y_{b+2},y_{b+1})$ is identified with $(a_0',a_1')$ or $(a_1',a_0')$ according as $o_1=1$ or $0$; in addition, a copy of $P_5^2$ is glued along its edge $x_1x_2$ to each of $a_1a_2$ and $a_2a_0$, with the orientations given by $o_2$ and $o_3$ as for $D(k;o)$, and a new vertex is joined to both ends of $a_1'a_2'$ and another to both ends of $a_2'a_0'$. (For $b=0$ the bridge is a single edge, which becomes a common side of the two triangles.) The graph $L^{\ast}(o)$ is defined in the same way with $b=2$, except that the bridge $y_1y_2y_3y_4$ is glued to the second triangle along $y_2y_4$ instead of $y_3y_4$, with $(y_2,y_4)$ identified with $(a_0',a_1')$ or $(a_1',a_0')$ according as $o_1=1$ or $0$; its two branch triangles then share the vertex $y_2$.

\begin{proof}[Proof of Theorem~\ref{thm:main}]
For $n\le5$ there is only one $n$-vertex MOP, so let $n\ge6$ and let $G\ne F_n$ be an $n$-vertex MOP. Starting from $G$, apply swaps (Lemma~\ref{lem:swap}), ear transfers (Lemma~\ref{lem:ear}), and transfers (Lemma~\ref{lem:transfer}) as long as one of them is available. Each step produces an $n$-vertex MOP with strictly smaller spectral radius, so the process stops at a MOP $G^*$ with $\rho(G^*)\le\rho(G)$, and equality holds only if $G^*=G$. At $G^*$ no operation applies: $G^*$ is swap-free, no vertex of degree $2$ lies in a triangle adjacent in $T(G^*)$ to a branch triangle, and no transfer satisfying the hypothesis of Lemma~\ref{lem:transfer} is available. By the second property, every chord that is a side of a branch triangle $\Delta$ cuts off, on the side away from $\Delta$, a MOP with at least two triangles, that is, with at least two vertices outside the chord; so the hypothesis $|V(Q^\circ)|\ge2$ of Lemma~\ref{lem:transfer} holds at every branch triangle of $G^*$. We prove that $\rho(G^*)>\rho(F_n)$ unless $G^*=F_n$. If $G^*=F_n$, then $G^*\ne G$ and $\rho(G)>\rho(G^*)=\rho(F_n)$. In either case $\rho(G)>\rho(F_n)$, which proves the theorem.

\begin{claim}[1]
If $L(G^*)=2$, then $G^*=F_n$.
\end{claim}

\noindent\textit{Proof.} By Lemma~\ref{lem:dual}(ii), $T(G^*)$ is a path. Suppose that $G^*$ has a vertex $p$ of degree $d\ge5$, with neighbors $q_1\cdots q_d$ as in Lemma~\ref{lem:MOP}(ii). The triangles $\{p,q_j,q_{j+1}\}$ form a subpath of $T(G^*)$, and for $2\le j\le d-2$ the triangle $\{p,q_j,q_{j+1}\}$ is joined to its two neighbors on this subpath through the chords $pq_j$ and $pq_{j+1}$; since $T(G^*)$ is a path, its side $q_jq_{j+1}$ is a boundary edge. Hence $q_3$ lies only in $\{p,q_2,q_3\}$ and $\{p,q_3,q_4\}$, so $\deg(q_3)=3$, while $p$ has the neighbors $q_1,q_2$ on one side of the chord $pq_3$ and the $d-3\ge2$ neighbors $q_4,\ldots,q_d$ on the other. This is a swap, which is impossible. Hence $\Delta(G^*)\le4$. Since $\sum_v(4-\deg(v))=4n-(4n-6)=6$ and the two vertices of degree $2$ contribute $4$, exactly two vertices have degree $3$ and all others have degree $4$. By Proposition~\ref{pro:Fn}, $G^*=F_n$. \hfill$\diamond$

\begin{claim}[2]
Let $\Delta$ be a branch triangle of $G^*$, and let $f$ be a side of $\Delta$ such that no branch triangle lies beyond $f$. Then the triangles beyond $f$ induce $P_{m+2}^2$ for some $m\ge3$, glued to $f$ along its edge $x_1x_2$.
\end{claim}

\noindent\textit{Proof.} The triangles beyond $f$ form a subtree of $T(G^*)$ without nodes of degree $3$, hence a path $S_1\cdots S_m$, where $S_1$ contains $f$ and $S_m$ is a leaf. A vertex of these triangles outside $f$ lies only in triangles beyond $f$, none of which is a branch triangle, so by Lemma~\ref{lem:skeleton}(i) it lies in at most three triangles; by Lemma~\ref{lem:skeleton}(ii), each end of $f$ lies in at most two of $S_1,\ldots,S_m$. By Lemma~\ref{lem:skeleton}(iii), the triangles $S_1,\ldots,S_m$ induce $P_{m+2}^2$ with $x_1x_2=f$. If $m=1$, then $S_1$ is a leaf adjacent to $\Delta$ and its vertex of degree $2$ admits an ear transfer. Suppose that $m=2$, and write $f=ab$ and $S_1=\{a,b,y\}$. The chord $S_1\cap S_2$ contains $y$ and one end of $f$, say $a$, so $S_2=\{a,y,z\}$ with $\deg(z)=2$; the side $yb$ of $S_1$ is a boundary edge, since $S_1$ is not a branch triangle, so $y$ lies only in $S_1$ and $S_2$ and $\deg(y)=3$. Then a transfer at $y$ satisfying the hypothesis of Lemma~\ref{lem:transfer} is available. Hence $m\ge3$. \hfill$\diamond$

\begin{claim}[3]
If $L(G^*)=3$, then $G^*\cong D(k;o)$ for some $o$ and some $k$ with $k_0,k_1,k_2\ge5$.
\end{claim}

\noindent\textit{Proof.} A tree with maximum degree at most $3$ has two more leaves than nodes of degree $3$, so $T(G^*)$ has exactly one branch triangle $\Delta=a_0a_1a_2$. By Claim~2, applied to each side $a_ia_{i+1}$, the triangles beyond $a_ia_{i+1}$ induce $P_{k_i}^2$ with $k_i\ge5$, glued along $x_1x_2$ with $\{x_1,x_2\}=\{a_i,a_{i+1}\}$. By Lemma~\ref{lem:MOP}(iii), these three blades meet only in vertices of $\Delta$. Choosing $o_i$ according to the order of $a_i,a_{i+1}$, we obtain $G^*\cong D(k;o)$. \hfill$\diamond$

\begin{claim}[4]
If $L(G^*)=3$, then $\rho(G^*)>\rho(F_n)$.
\end{claim}

\noindent\textit{Proof.} By Claim~3, $G^*\cong D(k;o)$ with all $k_i\ge5$, and $n=k_0+k_1+k_2-3$. For $p\in\{0,1,2\}$ and $j\ge5$, let $D_{o,p}(j)$ denote $D(k';o)$, where $k'_p=j$ and $k'_i=5$ for $i\ne p$. If $k_p\ge7$ for some $p$, then $D_{o,p}(7)$ is a subgraph of $G^*$, and the Appendix gives $\rho(D_{o,p}(7))>4$; hence $\rho(G^*)>4>\rho(F_n)$ by Lemma~\ref{lem:band}. Otherwise all $k_i\le6$, so $n\le15$ and $F_n$ is a subgraph of $F_{15}$. Since $D((5,5,5);o)$ is a subgraph of $G^*$ and the Appendix gives $\rho(D((5,5,5);o))>\rho(F_{15})$, we obtain
$$\rho(G^*)\ge\rho(D((5,5,5);o))>\rho(F_{15})\ge\rho(F_n).$$
\hfill$\diamond$

\begin{claim}[5]
If $L(G^*)\ge4$, then $\rho(G^*)>4$.
\end{claim}

\noindent\textit{Proof.} Now $T(G^*)$ has at least two branch triangles. Choose branch triangles $\Delta_1$ and $\Delta'$ at maximum distance in $T(G^*)$. If two of the three components of $T(G^*)-\Delta_1$ contained branch triangles, then one of these two components would not contain $\Delta'$; a branch triangle $\Delta''$ in it would be joined to $\Delta'$ by a path through $\Delta_1$, and so $\Delta''$ would be farther from $\Delta'$ than $\Delta_1$. Hence no branch triangle lies beyond two of the sides of $\Delta_1$. Let $\Delta_2$ be a branch triangle closest to $\Delta_1$, and let $\Delta_1T_1\cdots T_b\Delta_2$ be the path between them, so that $T_1,\ldots,T_b$ are not branch triangles and $\Delta_2$ lies beyond the third side of $\Delta_1$. By Claim~2, the first five vertices of the strip beyond each of the other two sides of $\Delta_1$ induce a copy of $P_5^2$ glued along $x_1x_2$.

Each side $f$ of $\Delta_2$ that is not on the path is a chord, so the other triangle containing $f$ supplies a vertex adjacent to both ends of $f$.

Next consider the bridge $T_1\cup\cdots\cup T_b$. If a vertex $v$ of a triangle $T_j$ lay in a branch triangle $\Delta_3\notin\{\Delta_1,\Delta_2\}$, the path in $T(G^*)$ formed by the triangles containing $v$ (Lemma~\ref{lem:dual}(i)) would leave the path $\Delta_1T_1\cdots T_b\Delta_2$ through a third chord of some $T_i$, making $T_i$ a branch triangle; so this does not happen unless $v\in\Delta_1\cup\Delta_2$. Hence every vertex of the bridge outside $\Delta_1\cup\Delta_2$ lies in at most three triangles, by Lemma~\ref{lem:skeleton}(i). By Lemma~\ref{lem:skeleton}(ii), applied to the path $T_1,\ldots,T_b$ from $\Delta_1$ and to the path $T_b,\ldots,T_1$ from $\Delta_2$, every vertex of $\Delta_1$ or $\Delta_2$ lies in at most two of $T_1,\ldots,T_b$. If $b\ge1$, Lemma~\ref{lem:skeleton}(iii) labels the vertices of the bridge as $x_1,\ldots,x_{b+2}$ with $T_j=\{x_j,x_{j+1},x_{j+2}\}$ and $\{x_1,x_2\}=\Delta_1\cap T_1$. The chord $\Delta_2\cap T_b$ is a side of $T_b$ other than $\{x_b,x_{b+1}\}$, so it is $\{x_{b+1},x_{b+2}\}$ or $\{x_b,x_{b+2}\}$. In the second case $x_b\in\Delta_2$, and if $b\ge3$, then $x_b$ lies in the three triangles $T_{b-2},T_{b-1},T_b$, contrary to Lemma~\ref{lem:skeleton}(ii) applied from $\Delta_2$. Hence the second case occurs only for $b\le2$; for $b=1$ it is the same as the first up to relabeling, and for $b=2$ it is the configuration of $L^\ast(o)$.

The pieces described above lie on different sides of the chords of $\Delta_1$ and $\Delta_2$, so by Lemma~\ref{lem:MOP}(iii) they meet only in vertices of $\Delta_1\cup\Delta_2$.

Suppose first that $b\ge3$. Then $\Delta_2\cap T_b=\{x_{b+1},x_{b+2}\}$; write $\Delta_2=\{x_{b+1},x_{b+2},t\}$. The side $x_{b+2}t$ of $\Delta_2$ is not on the path, and the triangle across it is not a leaf, since otherwise its vertex of degree $2$ would admit an ear transfer; write it as $\{x_{b+2},t,w\}$. Putting $x_{b+3}=t$ and $x_{b+4}=w$, the vertices $x_1,\ldots,x_{b+4}$ are distinct and span a copy of $P_{b+4}^2$ in which $x_1x_2$ is the side $\Delta_1\cap T_1$, since the additional edges $x_{b+1}x_{b+3}$, $x_{b+2}x_{b+3}$, $x_{b+2}x_{b+4}$, and $x_{b+3}x_{b+4}$ are sides of $\Delta_2$ and of $\{x_{b+2},t,w\}$. As $b+4\ge7$, the vertices $x_1,\ldots,x_7$, together with $\Delta_1$ and its two copies of $P_5^2$, form a copy of $D_{o',p}(7)$ for some $o'$ and $p$, and $\rho(D_{o',p}(7))>4$ by the Appendix. If $b\le2$, then $G^*$ contains a subgraph isomorphic to $L(b;o)$ or to $L^\ast(o)$ for some $o\in\{0,1\}^4$, and the Appendix gives $\rho(L(b;o))>4$ and $\rho(L^\ast(o))>4$ for all such $o$. In all cases $\rho(G^*)>4$. \hfill$\diamond$

By Claims~1, 4, and 5 and Lemma~\ref{lem:band}, $\rho(G^*)>\rho(F_n)$ unless $G^*=F_n$, as required.
\end{proof}

\section{Concluding Remarks}\label{sec:concluding}

The three operations in Section~\ref{sec:tools} lower the spectral radius for different reasons. The swap is decided by a single sign in the Schur complement onto the two ends of a chord, the ear transfer by comparing the Perron entries of two twin vertices in the resulting graph, and the transfer by a Schur complement onto three vertices, in which the piece that stays behind is compared with an ear. None of them needs any estimate of the gap $\rho(G)-\rho(F_n)$, which tends to $0$ for some families. Diagonal flips, which replace a chord by the other diagonal of its quadrilateral, are the more obvious local move, but they do not suffice: at $D((5,5,5);(1,1,1))$ every diagonal flip increases the spectral radius, and at $D((6,6,6);(1,1,1))$ no sequence of two flips decreases it. The operations used here act on whole pieces rather than on single chords, which is what allows the descent to reach $F_n$. The remaining comparison with $F_n$ is avoided for most graphs by the bound $\rho(F_n)<4$: once the operations are exhausted, every strip beyond a branch triangle that contains no further branch triangle has at least three triangles, two branch triangles then always produce spectral radius above $4$, and a single branch triangle does so unless all three strips are short, in which case the graph has at most $15$ vertices.

A natural next problem is the class of maximal planar graphs. For an $n$-vertex triangulation of the sphere, Euler's formula gives $3n-6$ edges, so the average degree tends to $6$, and one expects the minimizer to be as close to $6$-regular as the curvature condition $\sum_v(6-\deg(v))=12$ allows. An exhaustive enumeration of the triangulations with $6\le n\le12$ vertices supports this: in each case the minimizer is unique, has minimum degree at least $4$, and its degrees differ by at most $2$; it is the octahedron for $n=6$ and the icosahedron for $n=12$. The enumeration is carried out by the script \texttt{planar\_minimizers.py}, also provided as supplementary material.

\begin{que}\label{que:planar}
Which $n$-vertex maximal planar graphs have the minimum spectral radius?
\end{que}

In the planar case the weak dual is no longer a tree, so a chord no longer cuts the graph into two pieces, and the swap of Lemma~\ref{lem:swap} has no direct analogue. The twin structure used in Lemma~\ref{lem:ear}, on the other hand, is local and may still be available.

\section*{Declaration of generative AI and AI-assisted technologies in the manuscript preparation process}
During the preparation of this work, the author used Claude (Anthropic) for language refinement, technical editing, and computational verification of examples. The author reviewed and edited the output as needed and takes full responsibility for the content of the published article.

\section*{Appendix: the finite verification}

All inequalities used in Claims~4 and~5 are short enough to be checked by hand, and we list them below. The script \texttt{verify\_mop.py}, provided as supplementary material, rechecks them in exact integer arithmetic; it also checks Claim~5 directly for all orientation vectors and the identity~\eqref{eq:transfer} by symbolic expansion, and it runs in a few seconds on a laptop. Two kinds of certificates occur.
\begin{itemize}
\item[(a)] To show $\rho(H)>t$ for a rational $t$, we give a positive integer vector $x$ with $x^{\mathsf T}A(H)x>t\,x^{\mathsf T}x$.
\item[(b)] To show $\rho(H)\le t$, we give a positive integer vector $y$ with $(A(H)y)_i\le t\,y_i$ for every vertex $i$; this suffices by the Collatz--Wielandt inequality.
\end{itemize}
In the script, the vectors for (a) are obtained by scaling and rounding the Perron vector of $H$ and are then checked exactly, so a floating-point error can only make a check fail.

\medskip\noindent\textit{Claim 4.} Rotating the triangle $a_0a_1a_2$ gives
$$D(k_0,k_1,k_2;o_0,o_1,o_2)\cong D(k_1,k_2,k_0;o_1,o_2,o_0),$$
and reflecting it gives
$$D(k_0,k_1,k_2;o_0,o_1,o_2)\cong D(k_2,k_1,k_0;1-o_2,1-o_1,1-o_0).$$
Hence each of the $24$ graphs $D_{o,p}(7)$ is isomorphic to $D_{o',p'}(7)$ for one of the four pairs $(o',p')=((1,1,1),0)$, $((0,1,1),0)$, $((0,1,1),1)$, $((0,1,1),2)$, and each of the $8$ graphs $D((5,5,5);o)$ to $D((5,5,5);(1,1,1))$ or $D((5,5,5);(0,1,1))$. In the table, the vertices of $D(k;o)$ are listed as $a_0,a_1,a_2$, followed by $x_3,\ldots,x_{k_0}$ of the blade on $a_0a_1$, then by $x_3,\ldots,x_{k_1}$ of the blade on $a_1a_2$, and then by $x_3,\ldots,x_{k_2}$ of the blade on $a_2a_0$; the vertices of $F_{15}$ are listed as $x_1,\ldots,x_{15}$.

\begin{center}
{\footnotesize\begin{tabular}{lll}
\toprule
graph & certificate & value\\
\midrule
$D_{(1,1,1),0}(7)$ & (a): $(19,20,19,15,13,10,7,4,14,10,6,13,10,6)$ & $4642/1159>4$\\
$D_{(0,1,1),0}(7)$ & (a): $(10,8,8,7,6,4,3,2,6,4,2,6,5,3)$ & $941/234>4$\\
$D_{(0,1,1),1}(7)$ & (a): $(10,8,9,6,5,3,7,6,4,3,2,7,5,3)$ & $257/64>4$\\
$D_{(0,1,1),2}(7)$ & (a): $(10,8,9,6,5,3,6,4,2,7,6,5,3,2)$ & $994/247>4$\\
\midrule
$F_{15}$ & (b): $(13,21,29,36,42,46,49,50,49,46,42,36,29,21,13)$ & $\rho\le112/29$\\
$D((5,5,5);(1,1,1))$ & (a): $(10,10,10,7,5,3,7,5,3,7,5,3)$ & $722/183>112/29$\\
$D((5,5,5);(0,1,1))$ & (a): $(10,8,9,6,5,3,6,4,3,7,5,3)$ & $202/51>112/29$\\
\bottomrule
\end{tabular}}
\end{center}

\noindent Here the value in case (a) is the Rayleigh quotient $x^{\mathsf T}Ax/x^{\mathsf T}x$, and in case (b) it is $\max_i(Ay)_i/y_i$, attained at $x_3$ and $x_{13}$.

\medskip\noindent\textit{Claim 5.} Exchanging the names of $a_0'$ and $a_1'$ maps the second triangle together with its two attached triangles onto itself, so $L(b;o)$ does not depend on $o_1$ up to isomorphism. Exchanging the names of $a_0$ and $a_1$ gives $L(b;o_0,o_1,o_2,o_3)\cong L(b;1-o_0,o_1,1-o_3,1-o_2)$. The same holds for $L^\ast(o)$. Hence it suffices to treat $o_0=o_1=0$, which leaves $16$ graphs. In the table, their vertices are listed as $a_0,a_1,a_2$, followed by the vertices of the bridge other than $y_1,y_2$, then $a_2'$, then $x_3,x_4,x_5$ of the copy of $P_5^2$ on $a_1a_2$, then $x_3,x_4,x_5$ of the copy on $a_2a_0$, and finally the vertices joined to $a_1'a_2'$ and to $a_2'a_0'$. Each value is the Rayleigh quotient $x^{\mathsf T}Ax/x^{\mathsf T}x$ of the certificate of type (a), and all of them exceed $4$.

\begin{center}
{\footnotesize\begin{tabular}{llll}
\toprule
graph & $(o_2,o_3)$ & certificate (a) & value\\
\midrule
$L(0;o)$ & $(0,0)$ & $(6,7,6,5,5,3,2,4,3,2,3,3)$ & $926/231$\\
$L(0;o)$ & $(0,1)$ & $(5,5,4,4,3,2,1,3,2,1,2,2)$ & $237/59$\\
$L(0;o)$ & $(1,0)$ & $(5,5,6,4,4,3,2,4,3,2,2,2)$ & $337/84$\\
$L(0;o)$ & $(1,1)$ & $(7,6,6,5,4,3,2,5,3,2,3,3)$ & $926/231$\\
\midrule
$L(1;o)$ & $(0,0)$ & $(5,4,4,3,3,3,2,1,3,2,1,2,2)$ & $448/111$\\
$L(1;o)$ & $(0,1)$ & $(4,3,3,3,2,2,1,1,2,2,1,1,2)$ & $270/67$\\
$L(1;o)$ & $(1,0)$ & $(4,3,4,3,2,2,2,1,3,2,1,1,2)$ & $165/41$\\
$L(1;o)$ & $(1,1)$ & $(4,3,3,2,2,2,1,1,2,2,1,1,2)$ & $125/31$\\
\midrule
$L(2;o)$ & $(0,0)$ & $(4,4,4,3,3,2,3,2,1,3,2,1,1,1)$ & $201/50$\\
$L(2;o)$ & $(0,1)$ & $(4,3,3,3,3,2,2,2,1,2,2,1,1,1)$ & $153/38$\\
$L(2;o)$ & $(1,0)$ & $(4,3,4,3,3,2,3,2,1,3,2,1,1,1)$ & $374/93$\\
$L(2;o)$ & $(1,1)$ & $(4,3,3,3,2,2,2,2,1,2,2,1,1,1)$ & $286/71$\\
\midrule
$L^\ast(o)$ & $(0,0)$ & $(3,2,2,2,2,2,2,1,1,2,1,1,1,1)$ & $174/43$\\
$L^\ast(o)$ & $(0,1)$ & $(3,2,2,2,2,2,1,1,1,2,1,1,1,1)$ & $41/10$\\
$L^\ast(o)$ & $(1,0)$ & $(3,2,3,2,2,2,2,1,1,2,1,1,1,1)$ & $49/12$\\
$L^\ast(o)$ & $(1,1)$ & $(3,2,2,2,2,2,1,1,1,2,1,1,1,1)$ & $41/10$\\
\bottomrule
\end{tabular}}
\end{center}

\end{document}